\documentclass[a4paper,12pt]{amsart}

\usepackage[paper=a4paper,textwidth=400pt,textheight=620pt,includeheadfoot,twoside=True]{geometry}

\usepackage{amsmath,amssymb, amsthm,amsxtra}

\usepackage[mathscr]{eucal}

\usepackage{bm}
\usepackage[all]{xy}

\usepackage{aliascnt}

\theoremstyle{plain}
\newtheorem{thm}{Theorem}[section]
\newtheorem*{thm*}{Theorem}

\newaliascnt{prop}{thm}
\newaliascnt{cor}{thm}
\newaliascnt{lem}{thm}
\newaliascnt{claim}{thm}
\newaliascnt{defn}{thm}
\newaliascnt{ques}{thm}
\newaliascnt{conj}{thm}
\newaliascnt{fact}{thm}
\newaliascnt{rem}{thm}
\newaliascnt{ex}{thm}
\newtheorem{prop}[prop]{Proposition}

\newtheorem{lem}[lem]{Lemma}

\newtheorem*{prop*}{Proposition}
\newtheorem*{cor*}{Corollary}
\newtheorem*{lem*}{Lemma}
\newtheorem*{claim*}{Claim}
\theoremstyle{definition}
\newtheorem{defn}[defn]{Definition}
\newtheorem{ques}[ques]{Question}

\newtheorem*{defn*}{Definition}
\newtheorem*{ques*}{Question}
\newtheorem*{conj*}{Conjecture}

\newtheorem*{prob*}{Problem}

\newtheorem{rem}[rem]{Remark}
\newtheorem{ex}[ex]{Example}
\newtheorem*{fact*}{Fact}
\newtheorem*{rem*}{Remark}
\newtheorem*{ex*}{Example}
\aliascntresetthe{prop}
\aliascntresetthe{cor}
\aliascntresetthe{lem}
\aliascntresetthe{claim}
\aliascntresetthe{defn}
\aliascntresetthe{ques}
\aliascntresetthe{conj}
\aliascntresetthe{fact}
\aliascntresetthe{rem}
\aliascntresetthe{ex}
\usepackage[pointedenum]{paralist}
\usepackage{varioref}
\labelformat{equation}{\textnormal{(#1)}}
\labelformat{enumi}{\textnormal{(#1)}}

\setdefaultenum{(a)}{(i)}{(1)}{(A)}%

\usepackage[a4paper]{hyperref}

\def\textsectionN~{\textsection{}}

\renewcommand\phi{\varphi}
\renewcommand\epsilon{\varepsilon}
\renewcommand\leq{\leqslant}
\renewcommand\geq{\geqslant}

\makeatletter
\newcommand{\set}{%
  \@ifstar{\@setstar}{\@set}%
}%
\newcommand{\@setstar}[2]{\{\, #1 \mid #2 \,\}}
\newcommand{\@set}[1]{\{\, #1 \,\}}
\makeatother

\newcommand{\trans}[1][1]{\raisebox{#1ex}{\scriptsize\kern0.1em$t$\kern-0.1em}}

\newcommand{\PP}{\mathbb{P}}

\newcommand{\PN}{\PP^N}

\newcommand{\A}{\mathbb{A}}
\newcommand{\sO}{\mathscr{O}}
\newcommand{\ZZ}{\mathbb{Z}}

\newcommand{\CC}{\mathbb{C}}
\newcommand{\abs}[1]{\lvert #1 \rvert}

\newcommand{\textgene}[1]{\ \ \text{#1}\ \,}

\newcommand{\textand}{\textgene{and}}

\DeclareMathOperator{\pr}{pr}%
\DeclareMathOperator{\rk}{rk}%
\DeclareMathOperator{\codim}{codim}%
\DeclareMathOperator{\Univ}{Univ}%
\newcommand{\Gr}{\mathbb{G}}
\newcommand\spcirc{^\circ}
\newcommand{\tdiff}[2]{{\partial #1}/{\partial #2}}
\newcommand{\diff}{\tdiff}

\newcommand{\sH}{\mathscr{H}}
\newcommand{\sHd}{\mathscr{H}_d}
\newcommand{\Fh}{F_{\bm h}}
\newcommand{\shI}{\mathcal{I}}

\newcommand\Go{\Gr\spcirc}

\newcommand\RNi{1 \leq i \leq r}
\newcommand\RNk{0 \leq k \leq d^i}

\newcommand{\Da}{D_a}
\newcommand{\Db}{D_b}

\newcommand\matExa[2]{#1_{#2_1}(A_L) & #1_{#2_2}(A_L) & \cdots & #1_{#2_{N-1}}(A_L)}
\newcommand\ee{\bm{e}}

\title[Convex separably rationally connected complete intersections]
{Convex separably rationally connected \\ complete intersections}%

\email{katu@toki.waseda.jp}
\author[K.~Furukawa]{Katsuhisa~FURUKAWA}
\address{
  Department of Mathematics,
  School of Fundamental Science and Engineering,
  Waseda~University,
  Ohkubo~3-4-1, Shinjuku, Tokyo, 169-8555, Japan
}

\begin{document}

\maketitle

\begin{abstract}
  We give a generalization of a result of R.~Pandharipande to arbitrary characteristic:
  We prove that, if $X$ is a convex, separably rationally connected, smooth complete intersection in $\mathbb{P}^N$ over an algebraically closed field of arbitrary characteristic, then $X$ is rational homogeneous.
\end{abstract}

\section{Introduction}
\label{sec:introduction}

F.~Campana and T.~Peternell \cite{CP1} conjectured that, if $X$ is a smooth Fano variety (over $\CC$)
with nef tangent bundle, then $X$ is rational homogeneous.
This was answered over $\CC$ affirmatively by Campana and Peternell
in the case of $\dim(X) \leq 3$ \cite{CP1} and in the case of $\dim(X) = 4$ with Picard number $\rho_X > 1$ \cite{CP2},
by N.~Mok \cite{Mo} and J.~M.~Hwang \cite{Hw} for any $4$-dimensional $X$, and
by K.~Watanabe in the case of $\dim(X) = 5$ with $\rho_X > 1$ \cite{Wa}.

We say that a variety $X$ is \emph{convex} if every morphism $\mu: \PP^1 \rightarrow X$
satisfies $H^1(\PP^1, \mu^*T_X) = 0$.
R. Pandharipande \cite{Pa} proved that,
if $X$ is a convex, rationally connected,
smooth complete intersection in $\PN$ over $\CC$,
then $X$ is rational homogeneous.
Since a smooth Fano variety is rationally connected (over $\CC$),
and since nefness of the tangent bundle implies convexity,
the result solved the conjecture for complete intersections.

\begin{ques}\label{thm:ques}
  Does the statement of the above theorem of Pandharipande hold in arbitrary characteristic?
\end{ques}

In this paper, we prove:

\begin{thm}\label{thm:mainthm} Let $X$ be a convex, smooth complete intersection in $\PN$
  over an algebraically closed field of arbitrary characteristic.
  Assume that:
  \textnormal{(i)} there exists an immersion $\PP^1 \rightarrow X$.
  Then $X$ is of degree $\leq 2$;
  in particular, $X$ is rational homogeneous.
\end{thm}

Note that a complete intersection $X \subset \PN$ satisfies the above condition (i)
if one of the following conditions (ii-iv) holds:
\begin{inparaenum}[\normalfont (i)]
  \setcounter{enumi}{1}
\item $X$ is separably rationally connected;
\item $X$ is Fano;
\item $2N-2-r \geq \sum_{i=1}^r d^i$ holds, where $(d^1, \dots, d^r)$ is the type of $X$.
\end{inparaenum}
(See \autoref{thm:SRC-conditions}.)
As a corollary, a smooth Fano complete intersection with nef tangent bundle
is rational homogeneous in arbitrary characteristic.

The paper is organized as follows.
In \autoref{sec:param-space-compl},
we first study the parameter space $\sH$ of complete intersections of type $(d^1, \dots, d^r)$
and the incidence variety $I \subset \Gr(1, \PN) \times \sH$
parametrizing pairs $(L, X)$ such that $L \subset X$;
the techniques are based on \cite{BV}, \cite{DM}, \cite{Fu}, \cite[V. 4]{Ko}.
A rational curve $\mu: \PP^1 \rightarrow X$ is said to be \emph{free}
if $H^1(\PP^1, \mu^*T_X \otimes \sO_{\PP^1}(-1)) = 0$.
Next, we define the subset $J \subset I$ parametrizing pairs $(L, X)$ such that $L$ is non-free in $X$,
and we give a calculation method of the defining polynomials of $J$.
In \autoref{sec:constr-expect-pairs},
assuming all $d^i > 1$ and the product $\prod_{\RNi} d^i > 2$,
we construct a pair $(L, X)$ such that
the space of non-free lines in $X$ is smooth and of expected dimension at $L$.
In these constructions, it is necessary to take care of the characteristic two case (see \autoref{thm:rem-ch2}).
In consequence, we show that $\pr_2|_{J}: J \rightarrow \sH$ is surjective.
Thus, under our assumption,
every complete intersection $X$ of type $(d^1, \dots, d^r)$ has a non-free rational curve;
then we have the proof of \autoref{thm:mainthm} in \autoref{sec:existence-non-free}.

\subsection*{Acknowledgments}
The author would like to thank Kiwamu~Watanabe, who
asked \autoref{thm:ques} after his talk at Tokyo Denki University at August 2013.
Watanabe also asked
``Does there exist a counter-example of Campana-Peternell conjecture in positive characteristic?''.
The author was partially supported by JSPS KAKENHI Grant Number 25800030.

\section{Parameter space of complete intersections}\label{sec:param-space-compl}

\subsection{Incidence variety and its projection}

We fix some notation. Let $\sHd$ be the projectivization of the vector space $H^0(\PN, \sO(d))$,
whose general members parametrize hypersurfaces of degrees $d$ in $\PN$.
We take $\bm d = (d^1, \dots, d^r)$
with $r$ positive integers $d^1, \dots, d^r$, and denote by
$\abs{\bm d} := \sum_{\RNi}d^i$.
Let
\[
\sH := \sH_{d^1} \times \dots \times \sH_{d^r},
\]
whose general member $\bm h = (h^1, \dots, h^r)$ defines
an $(N-r)$-dimensional complete intersection $X \subset \PN$ of type $\bm d$.
For a homogeneous polynomial $\phi \in H^0(\PN, \sO(d))$ and a line $L \subset \PN$,
we denote by $\phi|_L$ the image of $\phi$ under the linear map $H^0(\PN, \sO(d)) \rightarrow H^0(L, \sO(d))$.

\begin{defn}
  We set
  \[
  I = \set*{(L, \bm h) \in \Gr(1, \PN) \times \sH}{h^1|_L = \dots = h^r|_L = 0},
  \]
  the incidence variety whose general members
  parametrize pairs
  $(L, X)$ such that $L \subset X$.
  We denote by
  \[
  \pr_1:I \rightarrow \Gr(1, \PN),\ \pr_2: I \rightarrow \sH,
  \]
  the first and second projections.
\end{defn}

\begin{rem}\label{thm:codimsIsH}
  Each fiber of $I \rightarrow \Gr(1, \PN)$ at $L$ 
  is isomorphic to $I_{d^1, L} \times \dots \times I_{d^r, L}$,
  where we denote by $\shI_L \subset \sO_{\PN}$ the ideal sheaf of $L \subset \PN$
  and by $I_{d, L} \subset \sHd$ the projectivization of
  $H^0(\PN, \shI_L(d))$. Note that $I_{d, L} \subset \sH_d$
  parametrizes hypersurfaces of degree $d$ containing the line $L$.
  From the exact sequence
  \[
  0 \rightarrow  H^0(\PN, \shI_L(d)) \rightarrow H^0(\PN, \sO_{\PN}(d)) \rightarrow H^0(L, \sO_L(d)) \rightarrow 0,
  \]
  we have $\codim(I_{d, L}, \sH_d) = d+1$. Hence
  \[
  \codim(I, \sH \times \Gr(1, \PN)) = \codim(I_{d^1, L} \times \dots \times I_{d^r, L}, \sH) = \abs{\bm{d}} + r.
  \]
  In addition, 
  \begin{equation}\label{eq:dimI-dimH}
    \dim(I) - \dim(\sH) = \dim \Gr(1, \PN) - (\abs{\bm{d}} + r)
    = 2N-2 -\abs{\bm d}-r.
  \end{equation}
\end{rem}

We denote by
\begin{equation}\label{eq:homog-coordi-PN}
  (S: T: Z_1: Z_2: \dots: Z_{N-1})
\end{equation}
the homogeneous coordinates on $\PN$.

\begin{defn}\label{thm:ci-delh}
  We take a line
  $L = (Z_1 = \dots = Z_{N-1} = 0) \subset \PN$.
  For a homogeneous polynomial $h \in H^0(\PN, \sO(d))$ with $h|_L = 0$,
  we define a homomorphism
  $\delta_L(h): \sO_L(1)^{\oplus N-1} \rightarrow \sO_L(d)$
  by
  \[
  (f_1, \cdots,  f_{N-1}) \mapsto
  h_{Z_1}|_L \cdot f_1 + \dots + h_{Z_{N-1}}|_L \cdot f_{N-1},
  \]
  where $h_{Z_j} := \diff{h}{Z_j} \in H^0(\PN, \sO(d-1))$.

  Now, we take $\bm h = (h^1, \dots, h^r) \in \sH$ such that
  $(L, \bm h) \in I$.
  Then,  $r$ homomorphisms $\delta_L(h^1), \dots, \delta_L(h^r)$ induce
  a homomorphism
  \[
  \delta_L(\bm h): \sO_L(1)^{\oplus N-1}
  \rightarrow \bigoplus_{\RNi} \sO_L(d^i) \cdot \epsilon^i.
  \]
\end{defn} 
\begin{rem}\label{thm:ci-delh-rem}
  Let $K$ be the ground field.
  The image of the linear map
  \[
  H^0(\delta_L(\bm h) (-1)):  H^0(L, \sO)^{\oplus N-1} = K^{\oplus N-1}
  \rightarrow \bigoplus_{\RNi} H^0(L, \sO(d^i-1)) \cdot \epsilon^i
  \]
  is equal to the vector subspace spanned by $N-1$ elements
  \begin{equation}\label{eq:sum-hiZ-epsilon}
    \sum_{\RNi} h^i_{Z_1}|_L\cdot \epsilon^i, \dots, \sum_{\RNi} h^i_{Z_{N-1}}|_L\cdot \epsilon^i.
  \end{equation}
\end{rem}

\begin{lem}\label{thm:delL-naturalmap}
  Let $(L, \bm h) \in I$ be as above.
  Let $X \subset \PN$ be a complete intersection defined by $\bm h$,
  and assume that $X$ is smooth along $L$. Then
  $\delta_L(\bm h)$ is regarded as
  the natural map $N_{L/\PN} \rightarrow N_{X/\PN}|_L$,
\end{lem}
\begin{proof}
  We first consider the case of $r=1$.
  The exact sequence $0 \rightarrow \shI_L^2 \rightarrow \shI_L \rightarrow N_{L/\PN}\spcheck \rightarrow 0$ induces
  the following linear map:
  \begin{equation}\label{eq:IL-Nv}
    H^0(\PN, \shI_L(d)) \rightarrow H^0(L, N_{L/\PN}\spcheck \otimes \sO(d)).
  \end{equation}
  Let $h \in H^0(\PN, \shI_L(d))$ and let $X \subset \PN$ be the hypersurface defined by $h$.
  Then the image of $h$ under the above map gives a natural homomorphism
  $N_{L/\PN} \rightarrow N_{X/\PN}|_L \simeq \sO_L(d)$ (see \cite[Rem.~3.9]{Fu}).

  Since $L = (Z_1 = \dots = Z_{N-1} = 0)$, we can write $h = h_1Z_1 + \dots + h_{N-1}Z_{N-1}$.
  Here the image of $h$ in $H^0(L, N_{L/\PN}\spcheck \otimes \sO(d))$ under \ref{eq:IL-Nv} is given by $(h_1|_L, \dots, h_{N-1}|_L)$.
  On the other hand, we have $h_{Z_i}|_L = h_i|_L$.
  Hence $N_{L/\PN} \rightarrow N_{X/\PN}$ is identified with $\delta_L(h): \sO_L(1)^{\oplus N-1} \rightarrow \sO_L(d)$.

  Next, for any $r \geq 1$,
  considering the linear map
  \begin{equation}\label{eq:ci-IL-Nv}
    H^0(\PN, \shI_L \otimes \bigoplus_{\RNi} \sO_{\PN}(d^i) ) \rightarrow H^0(L, N_{L/\PN}\spcheck \otimes \bigoplus_{\RNi} \sO_L(d^i)),
  \end{equation}
  instead of \ref{eq:IL-Nv},
  we have the assertion in a similar way to the case of $r=1$.
\end{proof}

\begin{rem}\label{thm:equiv-surj-nonfree}
  The linear map $H^0(\delta_L(\bm h) (-1))$ is surjective
  if and only if $L$ is free in $X$, i.e., $H^1(L, N_{L/X}(-1)) = 0$.
  This follows from $H^1(L, N_{L/\PN}(-1)) = 0$ and the following exact sequence:
  \[
  0 \rightarrow N_{L/X} \rightarrow N_{L/\PN} \xrightarrow{\delta_L(\bm h)}
  N_{X/\PN}|_L \rightarrow 0.
  \]
\end{rem}

\begin{lem}\label{thm:sIsH-surj}
  Assume $2N-2-r \geq \abs{\bm d}$. Then
  $\pr_2: I \rightarrow \sH$ is surjective,
  which means that every complete intersection of type $\bm d$ in $\PN$ contains a line.
\end{lem}
\begin{proof}
  It holds, for example, in a similar way to \cite[Proof of Prop. 3.15]{Fu}.
\end{proof}

\begin{defn}
  We denote by $J$
  the set of pairs $(L, \bm h) \in I$
  such that the linear map $H^0(\delta_L(\bm h) (-1))$ is \emph{not} surjective.
\end{defn}

Assume that a complete intersection $X \subset \PN$ defined by $\bm h \in \sH$
is smooth along a line $L$.
Then $(L, \bm h) \in J$ if and only if $L$ is non-free in $X$,
because of \autoref{thm:equiv-surj-nonfree}.

\begin{thm}\label{thm:surj-J-to-H}
  Assume that $2N-2-r \geq \abs{\bm d}$, $\prod_{\RNi}d^i > 2$, and all $d^i > 1$.
  Then
  $\pr_2|_{J}: J \rightarrow \sH$ is surjective.
\end{thm}

\begin{rem}\label{thm:J=I}
  If $N \leq \abs{\bm d}$, then we have $J = I$.
  This immediately follows from \autoref{thm:ci-delh-rem}.
  Thus $J \rightarrow \sH$ is surjective, because of \autoref{thm:sIsH-surj}.
\end{rem}

We denote by $J_{\bm h} := \pr_1(\pr_2^{-1}(\bm h) \cap J) \subset \Gr(1, \PN)$.
In the case of $\abs{\bm d} \leq N-2$,
the above theorem follows from:

\begin{prop}\label{thm:ci-constr-statement}
  Assume that $N-2 \geq \abs{\bm d}$, $\prod_{\RNi}d^i > 2$, and all $d^i > 1$.
  Then there exists a pair $(L, \bm h) \in J$ such that
  $J_{\bm h}$ is smooth and of dimension $N-r-2$
  at the point $L \in \Gr(1, \PN)$.
\end{prop}

Our goal is to construct
an expected pair $(L, \bm h)$, satisfying the statement of \autoref{thm:ci-constr-statement}.
In the next subsection, as a preparation,
we give defining polynomials of $J_{\bm h}$ on an affine open subset of $\Gr(1, \PN)$.

\subsection{Defining polynomials of the space of non-free lines}

Let us fix $(s:t)$ as the homogeneous coordinates on $\PP^1$, and 
$(S:T:Z_1:Z_2:\dots:Z_{N-1})$
on $\PN$.
We set $\Go_1 \subset \Gr(1, \PN)$ to be the standard open subset,
which is the set of lines $L$ such that $L \cap (S = T = 0) = \emptyset$.
Here $\Go_1 \simeq \A^{2(N-1)}$, which maps
$L \in \Go_1$ to
\begin{equation*}
  A_L =
  \begin{bmatrix}
    a_1, \dots, a_{N-1}
    \\b_1, \dots, b_{N-1}
  \end{bmatrix}
  \in \A^{2(N-1)}
\end{equation*}
if the line $L \subset \PN$ is equal to
the image of the morphism $\PP^1 \rightarrow \PN$ defined by
\begin{equation}\label{eq:line-aff-param}
  (s:t) \mapsto \begin{bmatrix} s & t  \end{bmatrix} \cdot \begin{bmatrix} E_2 & A_L  \end{bmatrix}
  = (s:t:sa_1+tb_1:\dots:sa_{N-1}+tb_{N-1}),
\end{equation}
where $E_2$ is the unit matrix of size $2$.
We set
\begin{equation*}
  \xi: \Go_1 \times \PP^1 \rightarrow \PN
\end{equation*}
to be the morphism which sends
$(L, (s:t)) \mapsto \begin{bmatrix} s & t  \end{bmatrix} \cdot \begin{bmatrix} E_2 & A_L  \end{bmatrix}$.

Let $X \subset \PN$ be a complete intersection
defined by homogeneous polynomials $\bm h = (h^1, \dots, h^r) \in \sH$.
The composite functions $h^i\circ \xi$ is regarded as a polynomial with variables
$s,t, a_1, \dots, a_{N-1}, b_1, \dots, b_{N-1}$.
Then $h^i\circ \xi$ is written as
\[
f^i_{0} s^{d^i} + f^i_1 s^{d^i-1}t + \dots + f^i_{d^i} t^{d^i},
\]
with some $f^i_0, f^i_1, \dots, f^i_{d^i} \in K[a_1, \dots, a_{N-1}, b_1, \dots, b_{N-1}]$.
Here $L \subset X$ if and only if
\[
f^i_k(A_L) = 0
\ \text{ for any $\RNi$ and $\RNk$.}
\]
Let $\Fh := \pr_1(\pr_2^{-1}(\bm h) \cap I) \subset \Gr(1, \PN)$,
which is equal to the set of lines lying in $X$.
Then $\Fh \cap \Go_1$ is equal to the zero set of $(r + \abs{\bm d})$-polynomials $f^i_k$'s in
the affine space $\Go_1 \simeq \A^{2(N-1)}$.

\begin{rem}\label{thm:deg-of-FXx}
  For a point $x$, we denote by $F_{\bm h, x}$ the set of $L \in \Fh$ such that
  $x \in L$.
  Let $x = (1: 0: \dots: 0)$. Then $F_{\bm h, x} \cap \Go_1$ is equal to the zero set of
  $f^i_k$'s and $a_1, a_2, \dots, a_{N-1}$ in $\Go_1 \simeq \A^{2(N-1)}$.
\end{rem}

For a homogeneous polynomial $h \in K[S,T,Z_1, \dots, Z_{N-1}]$ of degree $d$
with $h \circ \xi = f_0 s^d + f_1 s^{d-1}t + \dots + f_d t^d$,
we have
\begin{equation}\label{eq:diffF-1}
  \begin{aligned}
    (h \circ \xi)_{a_j} &= f_{0, a_j} s^d + f_{1, a_j} s^{d-1}t + \dots + f_{d, a_j} t^{d},
    \\
    (h \circ \xi)_{b_j} &= f_{0, b_j} s^d + f_{1, b_j} s^{d-1}t + \dots + f_{d, b_j} t^{d},
  \end{aligned}
\end{equation}
where $(h \circ \xi)_{a_j} := \diff{(h \circ \xi)}{a_j}$ and so on.
In addition, it follows that
\begin{equation}\label{eq:diffF-2}
  (h\circ \xi)_{a_j} = s \cdot (h_{Z_j}\circ \xi) \textand
  (h\circ \xi)_{b_j} = t \cdot (h_{Z_j}\circ \xi).
\end{equation}
This is shown as follows:
We write the morphism
$\xi$ by the coordinates $(\eta_0:\eta_1:\xi_1: \dots: \xi_{N-1})$,
which are given as in the right hand side of \ref{eq:line-aff-param}.
Then
$(h\circ \xi)_{a_j}$ is equal to $\diff{\eta_0}{a_j} \cdot (\diff{h}{S})\circ \xi + \diff{\eta_1}{a_j} \cdot (\diff{h}{T})\circ \xi + \sum_{1 \leq l \leq N-1} \diff{\xi_l}{a_j} \cdot (\diff{h}{Z_l})\circ \xi  = s \cdot (h_{Z_j}\circ \xi)$. Similarly, $(h\circ \xi)_{b_j}$ is also obtained.

We denote by $v(h_{Z_j}\circ\xi)$ the row vector of size $d$
which represents $h_{Z_j}\circ \xi$ with respect to
the basis
$s^{d-1}, s^{d-2}t, \dots, t^{d-1}$
of $H^0(\PP^1, \sO(d-1))$, i.e.,
\[
h_{Z_j}\circ \xi = v(h_{Z_j}\circ\xi) \cdot 
\trans\begin{bmatrix}
  s^{d-1}& s^{d-2}t& \dots& t^{d-1}
\end{bmatrix}.
\]
Each entry of $v(h_{Z_j}\circ\xi)$ is
a polynomial with variables $a_1, \dots, a_{N-1}, b_1, \dots, b_{N-1}$.

Now, for $\bm h = (h^1, \dots, h^r)$,
let us consider the $(N-1) \times \abs{\bm d}$ matrix
\begin{equation}\label{eq:defMhZ}
  M(\bm h) := \begin{bmatrix}
    v(h^1_{Z_1}\circ\xi) & \cdots & v(h^r_{Z_1}\circ\xi)\\
    \vdots&&\vdots\\
    v(h^1_{Z_{N-1}}\circ\xi) & \cdots & v(h^r_{Z_{N-1}}\circ\xi)
  \end{bmatrix},
\end{equation}
where it follows from \autoref{thm:ci-delh-rem} that $M(\bm h)(A_L)$ is of rank $\abs{\bm d}$ if and only if
$H^0(\delta_L(\bm h)(-1))$  is surjective.
Then
$J_{\bm h} \cap \Go_1$ is the zero set in $\Fh \cap \Go_1$ of the polynomials given
as the $\abs{\bm d} \times \abs{\bm d}$ minors of $M(\bm h)$.

For a polynomial $f$ with varieties $a_1, \dots, a_{N-1}, b_1, \dots, b_{N-1}$,
we set
\begin{gather*}
  \Da(f)_L := \begin{bmatrix}
    \matExa{f}{a}
  \end{bmatrix},
  \\
  \Db(f)_L := \begin{bmatrix}
    \matExa{f}{b}.
  \end{bmatrix}.
\end{gather*}
We denote by $\ee_i = \begin{bmatrix}
  e_{i,1} & e_{i,2} & \dots & e_{i,N-1}
\end{bmatrix}$
the row vector of size $N-1$ such that $e_{i,i} = 1$ and $e_{i,j} = 0$ for $i \neq j$,
and by $\bm 0$ the zero vector of size $N-1$.

\begin{rem}\label{thm:rk-matrixFG}
  Let $g_1, \dots, g_m$ be a minimal subset of the $\abs{\bm d} \times \abs{\bm d}$ minors of $M(\bm h)$
  such that $J_{\bm h} \cap \Go_1$ is locally defined in $\Go_1 \simeq \A^{2(N-1)}$ around $L$
  by the $\abs{\bm d}+r+m$ polynomials $g_1, \dots, g_m$ and $f^i_k$ ($\RNi,\, \RNk$). 
  We consider the $(\abs{\bm d}+r+m) \times 2(N-1)$ matrix
  \begin{equation}\label{eq:def-matDFG}
    \begin{bmatrix}
      \begin{pmatrix}
      \vdots & \vdots
      \\
      \Da(f^i_k) & \Db(f^i_k)
      \\
      \vdots & \vdots
    \end{pmatrix}
 \\
 \begin{pmatrix}
      \Da(g_1) & \Db(g_1)
      \\
      \vdots & \vdots
      \\
      \Da(g_{m}) & \Db(g_{m})
    \end{pmatrix}
    \end{bmatrix}_L.
  \end{equation}
  Then $J_{\bm h}$ is smooth at $L$ if and only if
  the rank of the above matrix \ref{eq:def-matDFG} is equal to $\abs{\bm d}+r+m$.
  We note that the rank can be calculated
  by using formulae \ref{eq:diffF-1}, \ref{eq:diffF-2}, \ref{eq:defMhZ}.
\end{rem}

\begin{lem}\label{thm:JinI-codim}
  Every irreducible component of $J$ is of codimension $\leq N-\abs{\bm d}$ in $I$.
\end{lem}
\begin{proof}
  We consider the morphism
  $\Phi: I \cap (\Go_1 \times \sH) \rightarrow \A^{(N-1)d}$ which sends $(L, \bm h) \mapsto M(\bm h)(A_L)$.
  Let $M_{\abs{\bm d}-1} \subset \A^{(N-1)d}$ the locus of matrix of rank $\leq {\abs{\bm d}-1}$.
  Then $J \cap (\Go_1 \times \sH)$ is equal to $\Phi^{-1}(M_{\abs{\bm d}-1})$.
  Since $M_{\abs{\bm d}-1}$ is of codimension $N-\abs{\bm d}$ in $\A^{(N-1)d}$ (see \cite[p. 67, II, \textsection 2, Prop.]{ACGH}),
  $J \cap (\Go_1 \times \sH)$ is of codimension $\leq N-\abs{\bm d}$ in $I$,
  which implies the assertion.
\end{proof}

\section{Construction of expected pairs}
\label{sec:constr-expect-pairs}

\subsection{Hypersurfaces}

In this subsection, we assume $r=1$
and give construction of pairs satisfying the statement of \autoref{thm:ci-constr-statement}.

\begin{rem}\label{thm:psi-zero}
  Let $(L, h)\in \Gr(1, \PN) \times \sHd$ satisfies $h|_L = 0$.
  Then the linear map $H^0(\delta_L(h)(-1)): K^{\oplus N-1} \rightarrow H^0(L, \sO(d-1))$
  given in \autoref{thm:ci-delh} is not surjective if and only if
  \[
  \psi(h_{Z_1}|_L) = \dots = \psi(h_{Z_{N-1}}|_L) = 0
  \]
  holds for some linear functional
  $\psi: H^0(\PP^1, \sO(d-1)) \rightarrow K$.
  A linear functional $\psi$ is represented by the row vector
  \begin{equation}\label{eq:psi-c}
    \begin{bmatrix}
      c_0 & c_1 & \cdots & c_{d-1}
    \end{bmatrix} \quad (c_0, \dots c_{d-1} \in K);
  \end{equation}
  this means that $\psi(s^{d-1-k}t^k) = c_k$ for each $k$.
\end{rem}

First we consider the case $(d, N) =  (4, 6)$. Let $(S: T: Z_1: \dots: Z_5)$
be the homogeneous coordinates on $\PP^6$, and $(s:t)$ be on $\PP^1$.

\begin{ex}\label{thm:hypsurf-4,6}
  Let $\psi: H^0(\PP^1, \sO(3)) \rightarrow K$ be a general linear functional,
  where we assume that $c_0 = 1$ in the expression \ref{eq:psi-c}.
  Let $X \subset \PP^6$
  be a hypersurface defined by the following homogeneous polynomial of degree $4$,
  \[
  h := h_1Z_1 + h_2Z_2 + h_3Z_3 + T^2Z_4Z_5,
  \]
  where we set
  \[
  h_j := c_jS^{3} - S^{3-j}T^{j} \quad (j = 1,2,3).
  \]
  Let $L = (Z_1 = \dots Z_5 = 0) \subset \PP^6$.
  Then $(L, h) \in J$, and
  $J_h$ is smooth and of dimension $N-3=3$ at the point $L \in \Gr(1, \PP^6)$.
  The reason is as follows.

  It follows that $h_{Z_j}\circ \xi = h_j \circ \xi = c_j s^3 - s^{3-j}t^j$ ($j = 1,2,3$).
  In particular, we find $(L, h) \in J$ as in \autoref{thm:psi-zero}.
  From \ref{eq:diffF-1}, \ref{eq:diffF-2},
  since $s \cdot h_{Z_j}\circ \xi = c_j s^4 - s^{4-j}t^j$, it follows that
  $f_{0,a_j} = c_j$, $f_{j,a_j} = -1$, and other $f_{j', a_j} = 0$.
  In addition, since $t \cdot h_{Z_j}\circ \xi = c_j s^3t - s^{3-j}t^{j+1}$, it follows that
  $f_{1,b_j} = c_j$, $f_{j+1,b_j} = -1$, and other $f_{j', b_j} = 0$.
  Hence,  we have
  \begin{equation*}
    \begin{bmatrix}
      D_a(f_0)\\
      D_a(f_1)\\
      D_a(f_2)\\
      D_a(f_3)\\
      D_a(f_4)
    \end{bmatrix}_L
    =
    \begin{bmatrix}
      c_1\ee_1 + c_2\ee_2 + c_{3}\ee_{3}
      \\
      -\ee_1
      \\
      -\ee_2
      \\
      -\ee_{3}
      \\
      \bm{0}
    \end{bmatrix},\
    \begin{bmatrix}
      D_b(f_0)\\
      D_b(f_1)\\
      D_b(f_2)\\
      D_b(f_3)\\
      D_b(f_4)
    \end{bmatrix}_L
    =
    \begin{bmatrix}
      \bm{0}
      \\
      c_1\ee_1 + c_2\ee_2 + c_{3}\ee_{3}
      \\
      -\ee_1
      \\
      -\ee_2
      \\
      -\ee_{3}
    \end{bmatrix}.
  \end{equation*}
  Here,
  $\begin{bmatrix}
    D_a(f_0)
    &
    D_b(f_0)
  \end{bmatrix}_L$
  is transformed to
  \[
  \begin{bmatrix}
    \bm{0}
    &
    (c_1^2-c_2) \ee_1 + (c_1c_2-c_3)\ee_2
  \end{bmatrix}
  \]
  under elementary row operations of
  $\begin{bmatrix}
    D_a(f_k)
    &
    D_b(f_k)
  \end{bmatrix}_L$
  with $1 \leq k \leq 4$.

  On the other hand, from \ref{eq:defMhZ}, we have
  \begin{equation}\label{eq:Mh-46}
    M(h) =
    \begin{bmatrix}
      c_1 & -1 \\
      c_2 && -1 \\
      c_3 &&& -1\\
      &&a_5& b_5\\
      &&a_4& b_4\\
    \end{bmatrix}
    \sim
    \begin{bmatrix}
      0 & -1 \\
      0 && -1 \\
      0 &&& -1\\
      c_2a_5+c_3b_5 && a_5& b_5\\
      c_2a_4+c_3b_4 &&a_4& b_4\\
    \end{bmatrix}
  \end{equation}
  where the right hand side is
  a transformation
  under elementary column operations, and
  where we consider blank entries in the matrix to be filled by $0$.
  Here
  $\rk M(h) < 4$
  if and only if $4 \times 4$ minors
  $g_1 = c_2a_4+c_3b_4$ and $g_2 = c_2a_5+c_3b_5$ are equal to zero.
  Now we have
  \[
  \begin{bmatrix}
    D_a(g_1)
    \\
    D_a(g_2)
  \end{bmatrix}_L
  =
  \begin{bmatrix}
    c_2\ee_4
    \\
    c_2\ee_5
  \end{bmatrix},\
  \begin{bmatrix}
    D_b(g_1)
    \\
    D_b(g_2)
  \end{bmatrix}_L
  =
  \begin{bmatrix}
    c_3\ee_4
    \\
    c_3\ee_5
  \end{bmatrix}.
  \]

  Hence the matrix \ref{eq:def-matDFG} in \autoref{thm:rk-matrixFG}
  is of rank $7$,
  which implies that
  $J_h \cap \Go_1$, the zero set of $7$ polynomials $f_0, \dots, f_4, g_1, g_2$,
  is smooth and of codimension $7$ in $\Go_1 \simeq \A^{10}$ at $L$.
\end{ex}

Recall that
$(S:T:Z_1:\dots:Z_{N-1})$ are the homogeneous coordinates on $\PN$.

\begin{ex}\label{thm:hypsurf-constr}
  Assume that $d \geq 3$ and $d \leq N-2$.
  Let $\psi: H^0(\PP^1, \sO(d-1)) \rightarrow K$ be a general linear functional
  with $c_0 = 1$ in \ref{eq:psi-c}.
  Let $X \subset \PN$
  be a hypersurface of degree $d$
  defined by the homogeneous polynomial
  $h := h_1Z_1 + h_2Z_2 + \dots + h_{d-1}Z_{d-1} + \tilde{h}$,
  where
  $h_j := c_{j}S^{d-1} - S^{d-j-1}T^{j} \quad (1 \leq j \leq d-1)$,
  and where
  \[
  \tilde h := 
  \begin{cases}
    T^{d-2}(Z_{d}Z_{d+1} + Z_{d+2}Z_{d+3}\dots + Z_{N-2}Z_{N-1})
    &\text{if $N-d$ is even,}
    \\
    ST^{d-3}Z_{d}Z_{d+1} + T^{d-2}(Z_{d+1}Z_{d+2} + Z_{d+3}Z_{d+4}\dots + Z_{N-2}Z_{N-1})
    &\text{if $N-d$ is odd.}
  \end{cases}
  \]
  Let $L = (Z_1 = \dots Z_{N-1} = 0)$.
  Then $(h, L) \in J$, and
  $J_h$ is smooth and of dimension $N-3$ at the point $L \in \Gr(1, \PN)$.
  It can be shown in a similar way to the previous example.
  
  We check the calculation of the dimension of $J_h$ and smoothness at $L$
  in the case where $N-d$ is odd.
  From \ref{eq:diffF-1}, \ref{eq:diffF-2}, we have
  \[
  \begin{bmatrix}
    D_a(f_0)\\
    D_a(f_1)\\
    \vdots\\
    D_a(f_{d-1})\\
    D_a(f_d)
  \end{bmatrix}_L
  =
  \begin{bmatrix}
    c_1\ee_1 + \dots + c_{d-1}\ee_{d-1}
    \\
    -\ee_1
    \\
    \vdots
    \\
    -\ee_{d-1}
    \\
    \bm{0}
  \end{bmatrix},\
  \begin{bmatrix}
    D_b(f_0)\\
    D_b(f_1)\\
    \vdots\\
    D_b(f_{d-1})\\
    D_b(f_d)
  \end{bmatrix}_L
  =
  \begin{bmatrix}
    \bm{0}
    \\
    c_1\ee_1 + \dots + c_{d-1}\ee_{d-1}
    \\
    -\ee_1
    \\
    \vdots
    \\
    -\ee_{d-1}
  \end{bmatrix}.
  \]
  From \ref{eq:defMhZ},
  we have that the $(N-1) \times d$ matrix
  $M(h)$ is
  \[
  \begin{bmatrix}
    c_1 & -1 & & &
    \\ 
    \vdots && \ddots
    \\
    c_{d-3} &&& -1 \\
    c_{d-2} &&&& -1 \\
    c_{d-1} &&&&& -1
    \\
    &&& a_{d+1} & b_{d+1}
    \\
    &&& a_{d} & b_{d} + a_{d+2} & b_{d+2}
    \\
    &&&& a_{d+1} & b_{d+1}
    \\
    &&&& a_{d+4} & b_{d+4}
    \\
    &&&& a_{d+3} & b_{d+3}
    \\
    &&&&    \vdots & \vdots
    \\
    &&&&    a_{N-1} & b_{N-1}
    \\
    &&&&    a_{N-2} & b_{N-2}
  \end{bmatrix}
  \sim
  \begin{bmatrix}
    0
    \\ 
    \vdots
    \\
    0 & - E_{d-1}
    \\
    0
    \\
    0 
    \\
    c_{d-3}a_{d+1}+c_{d-2}b_{d+1}
    \\
    c_{d-3}a_{d}+c_{d-2}(b_{d} + a_{d+2}) + c_{d-1}b_{d+2}
    \\
    c_{d-2}a_{d+1}+c_{d-1}b_{d+1}
    \\
    c_{d-2}a_{d+4}+c_{d-1}b_{d+4} & *
    \\
    c_{d-2}a_{d+3}+c_{d-1}b_{d+3}
    \\
    \vdots
    \\
    c_{d-2}a_{N-1}+c_{d-1}b_{N-1} 
    \\
    c_{d-2}a_{N-2}+c_{d-1}b_{N-2} 
  \end{bmatrix}.
  \]
  Hence $M(h)$ is of rank $< d$
  if and only if the $d \times d$ minors $g_1, \dots, g_{N-d}$ are zero,
  where
  $g_1 = c_{d-3}a_{d+1}+c_{d-2}b_{d+1}$,
  $g_2 = c_{d-3}a_{d}+c_{d-2}(b_{d} + a_{d+2}) + c_{d-1}b_{d+2}$,
  $g_3 = c_{d-2}a_{d+1}+c_{d-1}b_{d+1}$,
  and
  \[
  g_4 = c_{d-2}a_{d+3}+c_{d-1}b_{d+3},\
  g_5 = c_{d-2}a_{d+4}+c_{d-1}b_{d+4},
  \dots,
  g_{N-d} = c_{d-2}a_{N-1}+c_{d-1}b_{N-1}.
  \]
  Here we have:
  \[
  \begin{bmatrix}
    D_a(g_{1})\\
    D_a(g_{2})\\
    D_a(g_{3})\\
    D_a(g_{4})\\
    \vdots\\
    D_a(g_{N-d})\\
  \end{bmatrix}_L
  =
  \begin{bmatrix}
    c_{d-3}\ee_{d+1}
    \\
    c_{d-3}\ee_{d} + c_{d-2}\ee_{d+2}
    \\
    c_{d-2}\ee_{d+1}
    \\
    c_{d-2}\ee_{d+3}
    \\
    \vdots
    \\
    c_{d-2}\ee_{N-1}
  \end{bmatrix},\
  \begin{bmatrix}
    D_b(g_{1})\\
    D_b(g_{2})\\
    D_b(g_{3})\\
    D_b(g_{4})\\
    \vdots\\
    D_b(g_{N-d})\\
  \end{bmatrix}_L
  =
  \begin{bmatrix}
    c_{d-2}\ee_{d+1}
    \\
    c_{d-2}\ee_d + c_{d-1}\ee_{d+2}
    \\
    c_{d-1}\ee_{d+1}
    \\
    c_{d-1}\ee_{d+3}
    \\
    \vdots
    \\
    c_{d-1}\ee_{N-1}
  \end{bmatrix}.
  \]
  Since $c_1, \dots, c_{d-1} \in K$ are general,
  the matrix \ref{eq:def-matDFG} in \autoref{thm:rk-matrixFG}
  is of rank $N+1$,
  which implies that
  $J_h \cap \Go_1$   is smooth and of codimension $N+1$ in $\Go_1 \simeq \A^{2(N-1)}$ at the point $L$.
\end{ex}

\begin{rem}\label{thm:rem-ch2}
  Let the characteristic is \emph{not} equal to $2$.
  Then,
  by taking $\tilde h$ to be
  \[
  T^{d-2}(Z_{d}^2+\dots+Z_{N-1}^2),
  \]
  we can show the same statement of \autoref{thm:hypsurf-constr}
  with easier calculations.
  However, this does not work in characteristic $2$.
\end{rem}

\subsection{Complete intersections}

First we assume that $d^i \geq 3$ for some $i$.
Then \autoref{thm:hypsurf-constr} is generalized to the case of $r \geq 2$.
For example, we can calculate the case of $(r, N, (d^1, d^2)) = (2, 9, (4, 3))$,
as follows.

\begin{ex}\label{thm:ex-ci43}
  Let $\psi: H^0(\PP^1, \sO(3)) \rightarrow K$ be a general linear functional
  with $c_0 = 1$ in \ref{eq:psi-c}.
  Let $X \subset \PP^9$
  be a complete intersection
  defined by the following homogeneous polynomials $h^1, h^2$ of degrees $4$, $3$:
  \begin{align*}
    & h^1 := h^1_1Z_1 + h^1_2Z_2 + h^1_3Z_3 + T^2Z_4Z_5
    \quad (h^1_j := c_{j}S^{3} - S^{3-j}T^{j}),
    \\
    & h^2 := S^2Z_6 + TZ_7 + T^2Z_8.
  \end{align*}
  Let $L = (Z_1 = \dots Z_8 = 0) \subset \PP^9$.
  Then $(L, (h^1, h^2)) \in J$, and
  $J_{(h^1, h^2)}$ is smooth and of dimension $N-r-2=5$ at the point $L \in \Gr(1, \PP^9)$.
  The reason is as follows.

  It follows that $h^1_{Z_j} \circ \xi = h^1_j \circ \xi = c_js^3-s^{3-j}t^j$ ($j=1,2,3$). Thus
  we find that $H^0(\delta_L(h^1)(-1))$ is not surjective as in \autoref{thm:psi-zero}.
  This implies that $H^0(\delta_L(h^1, h^2)(-1))$ is also not surjective (see \autoref{thm:ci-delh-rem});
  hence $(L, (h^1, h^2)) \in J$.
  The calculation of $\Da(f^1_k)_L$ and $\Db(f^1_k)_L$ with $0 \leq k \leq 4$
  is the same as \autoref{thm:hypsurf-4,6}.
  On the other hand,
  from \ref{eq:diffF-1}, \ref{eq:diffF-2}, we have
  \begin{equation*}
    \begin{bmatrix}
      \Da(f^2_0)\\
      \Da(f^2_1)\\
      \Da(f^2_2)\\
      \Da(f^2_3)
    \end{bmatrix}_L
    =
    \begin{bmatrix}
      \ee_6
      \\
      \ee_7
      \\
      \ee_8
      \\
      \bm{0}
    \end{bmatrix},\
    \begin{bmatrix}
      \Db(f^2_0)\\
      \Db(f^2_1)\\
      \Db(f^2_2)\\
      \Db(f^2_3)
    \end{bmatrix}_L
    =
    \begin{bmatrix}
      \bm{0}
      \\
      \ee_6
      \\
      \ee_7
      \\
      \ee_8
    \end{bmatrix}.
  \end{equation*}
  Next, from \ref{eq:defMhZ}, we have
  \[
  M(h^1, h^2) =
  \begin{bmatrix}
    M' \\
    & E_3
  \end{bmatrix},
  \]
  where $M'$ is a $5 \times 5$ matrix which is the same as
  the left hand side of \ref{eq:Mh-46}.
  Thus, in a similar way to \autoref{thm:hypsurf-4,6},
  we find that $M(h^1, h^2)$ is of rank $< 7$
  if and only if the $7 \times 7$ minors $g_1, g_2$ are zero.
  As a result, the matrix \ref{eq:def-matDFG} in \autoref{thm:rk-matrixFG}
  is of rank $11$,
  which implies that
  $J_{(h^1, h^2)} \cap \Go_1$
  is smooth and of codimension $11$ in $\Go_1 \simeq \A^{16}$ at $L$.
\end{ex}

We denote by $\abs{\bm d}_{i_0} := \sum_{i = i_0}^{r} d^i$. Then $\abs{\bm d} = \abs{\bm d}_1$ and $\abs{\bm d}_{r} = d^r$.
\begin{ex}\label{thm:ci-constr}
  Assume that $d^1 \geq 3$ and $\sum d^i \leq N-2$.
  Let $\psi: H^0(\PP^1, \sO(d^1-1)) \rightarrow K$ be a general linear functional
  with $c_0 = 1$ in \ref{eq:psi-c}.
  Let $X \subset \PN$
  be a hypersurface defined by the following homogeneous polynomials
  $h^1, \dots, h^r$ of degrees $d^1, \dots, d^r$.
  We set
  $h^1 = h^1_1Z_1 + h^1_2Z_2 + \dots + h^1_{d^1-1}Z_{d^1-1} + \tilde h^1$,
  where
  $h^1_j := c_{j}S^{d^1-1} - S^{d^1-j-1}T^{j}$, and where
  \[
  \tilde h^1
  :=
  T^{d^1-2}(Z_{d^1}Z_{d^1+1} + Z_{d^1+2}Z_{d^1+3} + \dots + Z_{N-2-\abs{\bm d}_2}Z_{N-1-\abs{\bm d}_2})
  \]
  if $N - \abs{\bm d}$ is even, or
  \[
  \tilde h^1
  := ST^{d^1-3}Z_{d^1}Z_{d^1+1}
  + T^{d^1-2}(Z_{d^1+1}Z_{d^1+2} + Z_{d^1+3}Z_{d^1+4} + \dots + Z_{N-2-\abs{\bm d}_2}Z_{N-1-\abs{\bm d}_2})
  \]
  if $N - \abs{\bm d}$ is odd.
  Next we set $h^2, \dots, h^r$ by
  \begin{gather*}
    h^i = S^{d^i-1}Z_{N-\abs{\bm d}_i}+S^{d^i-2}TZ_{N+1-\abs{\bm d}_i}+\dots+T^{d^i-1}Z_{N-1-\abs{\bm d}_{i+1}}
    \quad (2 \leq i \leq r).
  \end{gather*}
  Let $L = (Z_1 = \dots Z_{N-1} = 0)$.
  Then, in a similar way to \autoref{thm:hypsurf-constr} and \autoref{thm:ex-ci43}, 
  one can show that $(L, (h^1, \dots, h^r)) \in J$ and that
  $J_{(h^1, \dots, h^r)}$ is smooth and of dimension $N-r-2$ at the point $L \in \Gr(1, \PN)$.
\end{ex}

\subsection{Complete intersections of hyperquadrics}
We complete the case where $r \geq 2$ and $d^1 = \dots = d^r = 2$.
For example, for $(N, r, (d^1, d^2)) = (7, 2, (2, 2))$, we have:

\begin{ex}\label{thm:ex-7222}
  Let $X \subset \PP^7$
  be the complete intersection defined by the following homogeneous quadric polynomials:
  \[
  h^1 = SZ_1 + TZ_2 + Z_5Z_6,\
  h^2 = SZ_2 + TZ_3 + Z_4Z_5.
  \]
  Let $L = (Z_1 = \dots Z_{6} = 0)$.
  Then $(L, (h^1, h^2)) \in J$, and
  $J_{(h^1, h^2)}$
  is smooth and of dimension $N-r-2 = 3$ at the point $L \in \Gr(1, \PP^7)$.
  The reason is as follows.

  It follows that
  \[
  [h^1_{Z_1}|_L,\ h^2_{Z_1}|_L] = [s,\ 0],\
  [h^1_{Z_2}|_L,\ h^2_{Z_2}|_L] = [t,\ s],\
  [h^1_{Z_3}|_L,\ h^2_{Z_3}|_L] = [0,\ t]
  \]
  and that other $[h^1_{Z_i}|_L,\ h^2_{Z_i}|_L]$'s are zero.
  Thus the above elements does not span the $4$-dimensional vector space
  $H^0(L, \sO(1)) \oplus H^0(L, \sO(1))$.
  Then, from \autoref{thm:ci-delh-rem}, we have $(L, (h^1, h^2)) \in J$.

  Next, we have
  \begin{gather*}
    \begin{bmatrix}
      \Da(f^1_0)
      \\
      \Da(f^1_1)
      \\
      \Da(f^1_2)
    \end{bmatrix}_L
    =
    \begin{bmatrix}
      \ee_1
      \\
      \ee_2
      \\
      \bm{0}
    \end{bmatrix},\
    \begin{bmatrix}
      \Db(f^1_0)
      \\
      \Db(f^1_1)
      \\
      \Db(f^1_2)
    \end{bmatrix}_L
    =
    \begin{bmatrix}
      \bm{0}
      \\
      \ee_1
      \\
      \ee_2
    \end{bmatrix},
    \\
    \begin{bmatrix}
      \Da(f^2_0)
      \\
      \Da(f^2_1)
      \\
      \Da(f^2_2)
    \end{bmatrix}_L
    =
    \begin{bmatrix}
      \ee_2
      \\
      \ee_3
      \\
      \bm{0}
    \end{bmatrix},\
    \begin{bmatrix}
      \Db(f^2_0)
      \\
      \Db(f^2_1)
      \\
      \Db(f^2_2)
    \end{bmatrix}_L
    =
    \begin{bmatrix}
      \bm{0}
      \\
      \ee_2
      \\
      \ee_3
    \end{bmatrix}.
  \end{gather*}
  On the other hand, we have
  \[
  M(h^1, h^2)
  =
  \begin{bmatrix}
    1 & 
    \\
    & 1 & 1 & 
    \\
    &  &  &  1
    \\
    &&a_5& b_5
    \\
    a_6& b_6 & a_4& b_4
    \\
    a_5 & b_5
  \end{bmatrix}
  \sim
  \begin{bmatrix}
    1 & 
    \\
    & 0 & 1 & 
    \\
    &  &  &  1
    \\
    &-a_5 &a_5& b_5
    \\
    a_6& b_6 -a_4 & a_4& b_4
    \\
    a_5 & b_5
  \end{bmatrix},
  \]
  where the right matrix is obtained by subtracting the third column from the second column.
  Hence
  $\rk(M(h^1, h^2)) < 4$
  if and only if $4 \times 4$ minors
  $g_1 = -a_5, g_2 = b_6-a_4, g_3 = b_5$ are zero.
  Here, it follows that
  \[
  \begin{bmatrix}
    \Da(g_1)
    \\
    \Da(g_2)
    \\
    \Da(g_3)
  \end{bmatrix}_L
  = 
  \begin{bmatrix}
    -\ee_5
    \\
    -\ee_4
    \\
    \bm{0}
  \end{bmatrix},\
  \begin{bmatrix}
    \Db(g_1)
    \\
    \Db(g_2)
    \\
    \Db(g_3)
  \end{bmatrix}_L
  =
  \begin{bmatrix}
    \bm{0}
    \\
    \ee_6
    \\
    \ee_5
  \end{bmatrix}.
  \]
  Therefore, the matrix \ref{eq:def-matDFG} in \autoref{thm:rk-matrixFG}
  is of rank $8$; then
  $J_{(h^1, h^2)} \cap \Go_1$ is smooth and of codimension $9$ in $\Go_1 \simeq \A^{12}$ at $L$.

\end{ex}

\begin{ex}\label{thm:exp-2..2}
  Assume that $r \geq 2$ and $2r \leq N-2$. 
  We set $X \subset \PN$
  to be the complete intersection defined by
  the following homogeneous quadric polynomials $h^1, \dots, h^r$; we set $h^1, h^2$ by
  \begin{gather*}
    h^1 = SZ_1 + TZ_2 + \tilde h^1,\
    h^2 = SZ_2 + TZ_3 + \tilde h^2,\
  \end{gather*}
  where
  \begin{align*}
    &\tilde h^1 := Z_{2r+1}Z_{2r+2} + \dots + Z_{N-2}Z_{N-1},\
    \tilde h^2 := Z_{2r}Z_{2r+1},
    &\text{ if $N-2r$ is odd,}
    \\
    &\tilde h^1 := Z_{2r}Z_{2r+1} + \dots + Z_{N-2}Z_{N-1},\
    \tilde h^2 := 0,
    &\text{ if $N-2r$ is even};
  \end{align*}
  and we set $h^3, \dots, h^r$ by
  \[
  h^3 = SZ_4 + TZ_5,\
  \dots,
  h^i = SZ_{2i-2} + TZ_{2i-1},\
  \dots,
  h^r = SZ_{2r-2} + TZ_{2r-1}.
  \]

  Let $L = (Z_1 = \dots Z_{N} = 0)$.
  Then $(L, (h^1, \dots, h^r)) \in J$, and
  $J_{(h^1, \dots, h^r)}$ is smooth and of dimension $N-r-2$ at the point $L \in \Gr(1, \PN)$.
  This assertion follows in a similar way to \autoref{thm:ex-7222}.
\end{ex}

\subsection{Surjectivity of the projection from $J$ to $\sH$}

First, we consider the remained part, $\abs{\bm d} = N-1$:
\begin{lem}\label{thm:etale}
  If $\abs{\bm d} = N-1$, then the statement of \autoref{thm:surj-J-to-H} holds.
\end{lem}

\begin{rem}
  Since all $d^i \geq 2$, we have $2r \leq \abs{\bm d} < N$.
  Hence $2(N-r) = 2N -2r > N$.
  Then \cite[Cor. 2]{FH} implies that
  a smooth complete intersection $X \subset \PN$ of type $\bm d$ is simply connected.
\end{rem}

\begin{proof}[Proof of \autoref{thm:etale}]
  Assume that $\pr_2(J) \neq \sH$, and take a general $\bm h \in \sH$ such that $\bm h \notin \pr_2(J)$.
  Let $X \subset \PN$ be a complete intersection defined by $\bm h$.
  We denote by $\Fh \subset \Gr(1, \PN):= \pr_1(\pr_2^{-1}(\bm h))$,
  which is the space of lines lying in $X$,
  and consider universal family $\Univ \subset \Fh \times X$ with projection $u: \Univ \rightarrow X$.
  Here $u^{-1}(x)$ is identified with the space of lines lying in $X$ passing through a point $x$.
  Since $\pr_2^{-1}(\bm h) \cap J = \emptyset$, every line $L \in \Fh$ is free, where we have $N_{L/X} = \sO_L^{\oplus N-1}$.
  Since $h^1(L, N_{L/X}(-1)) = 0$,
  every fiber $u^{-1}(x)$ is smooth and of dimension $h^0(L, N_{L/X}(-1)) = 0$.
  Hence $u$ is \'etale morphism.
  Since $X$ is simply connected, $u$ is isomorphic.
  This contradicts that the length of $u^{-1}(x)$ is $> 1$ if $X$ is general (see \autoref{thm:deg-of-FXx}).
\end{proof}

\begin{proof}[Proof of \autoref{thm:ci-constr-statement}]
  It follows from Examples \ref{thm:ci-constr} and \ref{thm:exp-2..2}.
\end{proof}

Now, we have the surjectivity of $\pr_2|_J: J \rightarrow \sH$.

\begin{proof}[Proof of \autoref{thm:surj-J-to-H}]
  The case of $\abs{\bm d} \geq N-1$
  follows immediately from \autoref{thm:J=I} and \autoref{thm:etale}.

  Assume $\abs{\bm d} \leq N-2$.
  From \autoref{thm:ci-constr-statement},
  there exists a pair $(L, \bm h) \in J$ such that
  $\pr_2^{-1}(\bm h) \cap J \simeq J_{\bm h}$
  is smooth and of dimension $N-r-2$ at $(L, \bm h)$.
  We consider the exact sequence
  \[
  0 \rightarrow T_{(L, \bm h)} (\pr_2^{-1}(\bm h) \cap J)
  \rightarrow T_{(L, \bm h)} J
  \xrightarrow{d_{(L, \bm h)} (\pr_2|_{J})} T_{\bm h}\sH.
  \]
  From \autoref{thm:JinI-codim} and the equality~\ref{eq:dimI-dimH} in \autoref{thm:codimsIsH},
  we have
  \begin{multline*}
    \dim (T_{(L, \bm h)} J) - \dim (T_{(L, \bm h)} (\pr_2^{-1}(\bm h)\cap J))
    \\
    \geq (\dim (I) - (N - \abs{\bm d})) - (N-r-2) = \dim (\sH),
  \end{multline*}
  which implies that
  the linear map $d_{(L, \bm h)}(\pr_2|_{J})$ is surjective.
  Hence the morphism $\pr_2|_{J}$ is smooth on an open neighborhood of $(L, \bm h)$ in $J$, and hence, is surjective.
\end{proof}

\section{Non-free rational curves}
\label{sec:existence-non-free}

From \autoref{thm:surj-J-to-H}, there exists a non-free line lying in a complete intersection $X \subset \PN$
in the case of $2N-2-r \geq \abs{\bm d}$ and $\deg(X) > 2$. On the other hand,
for $\abs{\bm d} > N$, the following lemma is immediately obtained:

\begin{lem}\label{thm:X-d>N-nonfree}
  Let $X \subset \PN$ be a smooth complete intersection of type $\bm d = (d_1, \dots, d_r)$.
  If $\abs{\bm d} > N$, then every immersion $\mu: \PP^1 \rightarrow X$ is non-free.
\end{lem}
\begin{proof}
  Let 
  $\mu^*T_X \simeq \bigoplus_{i=1}^{N-r} \sO_{\PP^1} (a_i)$ be the splitting on $\PP^1$
  with integers $a_1 \geq a_2 \geq \cdots \geq a_{N-r}$.
  Taking the positive integer $b$ with $\mu^*(\sO_X(1)) = \sO_{\PP^1}(b)$,
  since
  $\bigwedge^{N-r} T_X =  \sO_X(N+1-\abs{\bm d})$,
  we have $\sum_{1 \leq i \leq N-r}a_i = b(N+1-\abs{\bm d}) \leq 0$.
  Since $T_{\PP^1} \rightarrow \mu^*T_X$ is injective, we find that $a_1 \geq 2$.
  Therefore $a_{i_0} < 0$ for some $i_0$. Then $H^1(\PP^1, \mu^*T_X \otimes \sO_{\PP^1}(-1)) \neq 0$.
\end{proof}

\begin{rem}\label{thm:SRC-conditions}
  We check that each of the conditions (ii-iv) in \autoref{sec:introduction} implies
  the condition (i) in \autoref{thm:mainthm}.
  First, (iv) implies (i); this is because, $X$ contains a line if (iv) holds, as in \autoref{thm:sIsH-surj}.
  Next, (ii) or (iii) implies (iv), as follows.
  From \cite[IV, Thm.~3.7]{Ko},
  $X$ is separably rationally connected if and only if
  there exists $\mu: \PP^1 \rightarrow X$ such that $\mu^*T_X$ is ample,
  where the latter condition implies $\abs{\bm d} \leq N$.
  On the other hand, $X$ is Fano if and only if $\abs{\bm d} \leq N$.
\end{rem}

\begin{proof}[Proof of \autoref{thm:mainthm}]
  Let $X \subset \PN$ be a complete intersection of type $\bm d = (d^1, \dots, d^r)$
  satisfying the assumption of \autoref{thm:mainthm}.
  Without loss of generality,
  we may assume that all $d^i > 1$.

  Suppose that $\prod d^i =\deg(X) > 2$.
  Then we can find a non-free immersion $\mu: \PP^1 \rightarrow X$, as follows.
  Assume $2N-2-r \geq \abs{\bm d}$.
  Taking defining polynomials $\bm h = (h^1, \dots, h^r)$ of $X$,
  from \autoref{thm:surj-J-to-H},
  we have that
  $J \cap \pr_2^{-1}(\bm h) \neq \emptyset$;
  this means that a non-free line in $X$ exists.
  Assume $2N-2-r < \abs{\bm d}$.
  We may also assume $N-r \geq 2$. Then $N < \abs{\bm d}$ holds.
  By the condition (i) in \autoref{thm:mainthm},
  an immersion exists and is non-free because of \autoref{thm:X-d>N-nonfree}.

  Let $\mu: \PP^1 \rightarrow X$ be a non-free immersion.
  For the splitting
  $\mu^*T_X \simeq \bigoplus_{i=1}^{N-r} \sO_{\PP^1} (a_i)$
  with $a_i \in \ZZ$, it follows $a_{i_0} \leq -1$ for some $i_0$.
  Taking a double cover $\iota_2: \PP^1 \rightarrow \PP^1$,
  since the splitting of $(\iota_2 \circ \mu)^* T_X$ contains $\sO_{\PP^1}(2a_{i_0})$,
  we have $H^1(\PP^1, (\iota_2 \circ \mu)^* T_X) \neq 0$,
  which contradicts that $X$ is convex.
  Hence $\deg(X) \leq 2$ must hold.
\end{proof}

\vspace{1ex}%


\begin{thebibliography}{99}

\bibitem{ACGH}
  E. Arbarello, M. Cornalba, P. A. Griffiths, and J. Harris, Geometry of Algebraic Curves,
  Volume I. Springer 1985.

\bibitem{BV}
  W.~Barth and A.~Van~de Ven.
  \newblock Fano varieties of lines on hypersurfaces.
  \newblock \emph{Arch. Math. (Basel)} {\bfseries 31} (1978/79), 96--104.

\bibitem{CP1} F. Campana and T. Peternell, Projective manifolds whose tangent bundles are numerically effective,
  Math. Ann. \textbf{289} (1991), 169-187.
\bibitem{CP2} F. Campana, T. Peternell, 4-folds with numerically effective tangent bundles and second Betti
  numbers greater than one, Manuscripta Math. \textbf{79} (1993), no. 3-4, 225-238.

\bibitem{DM}
  O.~Debarre and L.~Manivel.
  \newblock Sur la vari\'et\'e des espaces lin\'eaires contenus dans une
  intersection compl\`ete.
  \newblock \emph{Math. Ann.} {\bfseries 312} (1998), 549--574.

\bibitem{FH}
  W.~Fulton and J.~Hansen,
  A connectedness theorem for projective varieties,
  with applications to intersections and singularities of mappings,
  Ann. Math. {\bf 110} (1979) 159--166.

\bibitem{Fu} K.~Furukawa, Rational curves on hypersurfaces,
  J. reine angew. Math. \textbf{665} (2012), 157-188.

\bibitem{Hw} J. M. Hwang, Rigidity of rational homogeneous spaces, Proceedings of ICM. 2006 Madrid, vol-
  ume II, European Mathematical Society, 2006, 613-626.

  
\bibitem{Ko} J.~Kollar, Rational curves on algebraic varieties, Ergeb. Math. Grenzgeb. (3) \textbf{32},
  Springer-Verlag, Berlin 1996.

\bibitem{Mo} N. Mok, On Fano manifolds with nef tangent bundles admitting 1-dimensional varieties of
  minimal rational tangents, Trans. Amer. Math. Soc. \textbf{354} (2002), no. 7, 2639-2658.

\bibitem{Pa} R. Pandharipande, Convex rationally connected varieties,
  Proc. Amer. Math. Soc. \textbf{141} (2013), 1539-1543.


\bibitem{Wa} K.~Watanabe, Fano 5-folds with nef tangent bundles and Picard numbers greater than one,
  Math. Z. \textbf{276} (2014), 39--49.

\end{thebibliography}
\end{document}